\documentclass[leqno,11pt, a4]{amsart}
\usepackage[utf8]{inputenc}
\usepackage{amsthm}
\usepackage{amsmath}
\usepackage{enumitem}
\usepackage{amsfonts}
\usepackage{amssymb}
\usepackage{graphicx}
\usepackage{pdfsync}
\newtheorem{theorem}{Theorem}[section]
\newtheorem{definition}{Definition} [section]
\newtheorem{corollary}{Corollary}[theorem]
\newtheorem{lemma}[theorem]{Lemma}
\newtheorem{proposition}[theorem]{Proposition}

\newcounter{yuppo}
\newtheorem{yuppi}{Theorem}

\newcommand{\Keler} {K\"{a}hler }

\newcommand{\End}{\operatorname{End}}

\newcommand{\cd}{\cdot}

\newcommand{\om}{\omega}

\renewcommand{\phi}{\varphi}
\newcommand{\cinf}{C^\infty}

\newcommand{\ra}{\rightarrow}

\newcommand{\C}{\mathbb{C}}
\newcommand{\R}{\mathbb{R}}

\newcommand{\Gl}{\operatorname{Gl}}

\newcommand{\ad}{{\operatorname{ad}}}

\newcommand{\ga}{\gamma}

\newcommand{\enf}{\emph}

\newcommand{\liu}{\mathfrak{u}}

\newcommand{\lia}{\mathfrak{a}}

\newcommand{\lier}{\mathfrak{r}}

\newcommand{\lieg}{\mathfrak{g}}

\newcommand{\liep}{\mathfrak{p}}

\newcommand{\la}{\lambda}

\newcommand{\alfa}{\alpha}
 \newcommand{\vacuo}{\emptyset}
\newcommand{\OO}{\mathcal{O}} 
\newcommand{\ext}{\operatorname{ext}} 
\newcommand{\sx}{\langle} 
\newcommand{\xs}{\rangle}
\newcommand{\relint}{\operatorname{relint}} 

\newcommand{\Crit}{\operatorname{Crit}} 

\newcommand{\noparty}[1]{}
\newcommand{\changed}[1]{{#1}}

\newcommand{\scalo}{\sx \, , \, \xs}

\newcommand{\metrica}{(\, , \, )}

\newcommand{\mup}{\mu_\liep}
\newcommand{\mupb}{\mu_\liep^\beta}

\title{Compact orbits of Parabolic subgroups}
\author{Biliotti, L.}
\address{Leonardo Biliotti, Dipartimento di Scienze Matematiche, Fisiche e Informatiche \\
          Universit\`a di Parma (Italy)}
\email{leonardo.biliotti@unipr.it}
\author{Windare, O.J.}
\address{Oluwagbenga Joshua Windare, Dipartimento di Scienze Matematiche, Fisiche e Informatiche \\
          Universit\`a di Parma (Italy)}
\email{oluwagbengajoshua.windare@unipr.it}
\thanks{The first author was partially supported by the Project PRIN 2015, ``Real and Complex Manifolds: Geometry, Topology and Harmonic Analysis'',
Project PRIN  2017 ``Real and Complex Manifolds: Topology, Geometry and holomorphic dynamics'' and by GNSAGA INdAM}

\begin{document}
\maketitle
\begin{abstract}
\noindent
We study the action of a real reductive group $G$ on a real submanifold $X$ of a \Keler manifold $Z$.
We suppose that the action of a compact connected Lie group $U$ with Lie algebra $\mathfrak{u}$ extends holomorphically to an action of the complexified group
$U^\C$ and that the $U$-action on $Z$ is Hamiltonian. If $G\subset U^\C$ is compatible there exists a gradient map
$\mu_{\mathfrak p}:X \longrightarrow \mathfrak p$ where $\mathfrak g=\mathfrak k \oplus \mathfrak p$ is a Cartan decomposition of $\mathfrak g$. In this paper we describe compact orbits of parabolic subgroups of $G$ in term of the gradient map $\mu_\mathfrak p$.
\end{abstract}
\section{Introduction}
\pagenumbering{arabic}
In this paper, we study the actions of real reductive groups on real submanifolds of \Keler manifolds.

Let $U$ be a compact connected Lie group with Lie algebra $\mathfrak{u}$ and let $U^\C$ be its complexification. We say that a subgroup $G$ of $U^\C$
is compatible if $G$ is closed and the map $K\times \mathfrak{p} \to G,$ $(k,\beta) \mapsto k\text{exp}(\beta)$ is a diffeomorpism where $K := G\cap U$ and
$\mathfrak{p} := \mathfrak{g}\cap \text{i}\mathfrak{u};$ $\mathfrak{g}$ is the Lie algebra of $G.$ The Lie algebra $\mathfrak{u}^\C$ of $U^\C$
is the direct sum $\mathfrak{u}\oplus i\mathfrak{u}.$ It follows that $G$ is compatible with the Cartan decomposition
$U^\C = U\text{exp}(\text{i}\mathfrak{u})$, $K$ is a maximal compact subgroup of
$G$ with Lie algebra $\mathfrak{k}$ and that $\mathfrak{g} = \mathfrak{k}\oplus \mathfrak{p}.$ Note that $G$ has finitely number of connected components. In the sequel we always assume that $G$ is connected.

Let $(Z,\omega)$ be a \Keler manifold with an holomorphic action of the complex reductive group $U^\C$.  We also assume $\omega$ is $U$-invariant and that there is a $U$-equivariant momentum map $\mu : Z \to \mathfrak{u}^*.$ By definition, for any $\xi \in \mathfrak{u}$ and $z\in Z,$ $d\mu^\xi = i_{\xi_Z}\omega,$ where $\mu^\xi(z) := \sx\mu(z), \xi\xs$, and $\xi_Z$ denotes the fundamental vector field induced on $Z$ by the action of $U,$
$$
\xi_Z(z) := \frac{d}{dt}|_{t=0}\text{exp}(t\xi)\cdot z.
$$
The inclusion i$\mathfrak{p}\hookrightarrow \mathfrak{u}$ induces by restriction, a $K$-equivariant map $\mu_{\text{i}\mathfrak{p}} : Z \to (\text{i}\mathfrak{p})^*.$ Using an  $\mathrm{Ad}(U)$-invariant inner product on $u$ to identify $(\text{i}\mathfrak{p})^*$ and $\mathfrak{p},$ $\mu_{\text{i}\mathfrak{p}}$ can be viewed as a map $\mu_{\mathfrak{p}} : Z \to \mathfrak{p}.$ For $\beta \in \mathfrak{p}$ let $\mu_\mathfrak{p}^\beta$ denote $\mu^{-\text{i}\beta}.$ i.e., $\mu_\mathfrak{p}^\beta(z) := -\sx\mu(z), i\beta\xs.$  Then the grad$\mu_\mathfrak{p}^\beta = \beta_Z$ where grad is computed with respect to the Riemannian metric induced by the \Keler structure. The map $\mu_\mathfrak{p}$ is called the gradient map associated with $\mu$.  For a $G$-stable locally closed real submanifold $X$ of $Z,$ we consider $\mu_\mathfrak{p}$ as a mapping $\mu_\mathfrak{p} : X\to \mathfrak{p}$ (see  \cite{heinzner-schwarz-stoetzel} for more details).

Let $\lia \subset \liep$ be a Abelian subalgebra and let $\pi:\liep \longrightarrow \lia$ be the orthogonal projection onto $\lia$. Then $\mu_{\mathfrak a}=\pi_\lia \circ \mu_{\mathfrak p}$ is the gradient map associated to $A=\exp(\lia)$.

If $\beta
\in \liep$ set
\begin{gather*}
\begin{gathered}
  G^{\beta+} :=\{g \in G : \lim_{t\to - \infty} \text{exp} ({t\beta}) g
  \text{exp} ({-t\beta}) \text { exists} \}\\
  R^{\beta+} :=\{g \in G : \lim_{t\to - \infty} \text{exp} ({t\beta})
  g \text{exp} ({-t\beta}) =e \}\\
  G^\beta=\{g\in G:\, \mathrm{Ad}(g)(\beta)=\beta\}
\end{gathered}
\qquad
  \lier^{\beta+}: = \bigoplus_{\la > 0} V_\la (\ad \beta).
\end{gather*}
Note that $ \lieg^{\beta+}= \lieg^\beta \oplus \lier^{\beta+}$. It is well-known that
$G^{\beta +} $ is a parabolic subgroup of $G$ with Lie algebra $\lieg^{\beta+}$ and  every parabolic subgroup of $G$ arises as $G^{\beta+}$ for some $\beta \in \liep$.  $R^{\beta+}$ is connected and it is the unipotent radical of $G^{\beta+}$. $G^\beta$ is a Levi factor of $G^{\beta+}$ (see \cite[Lemma 9]{LA} and \cite{borel-ji-libro} for more details).
Our first main result is the following.
\begin{yuppi}\label{main1}
Let $\beta \in \mathfrak{p}.$ Then:
\begin{itemize}
\item if $G^{\beta+}\cdot x$ is compact, then $\OO = G\cdot x$ is compact and $G^{\beta+}\cdot x$ is a finite union of connected components of $\mathrm{max}_{\OO}(\beta)=\left\{p\in \OO:\, \mathrm{max}_{z\in \OO} \, \mu_{\mathfrak p}^\beta =\mu_{\mathfrak p}^\beta (p)\right\};$
\item if $\OO$ is a compact $G$-orbit, then $\mathrm{max}_{\OO}(\beta)$ is a finite union of compact $G^{\beta+}$-orbits.
\end{itemize}
In particular, the number of compact $G^{\beta +}$-orbits is equal or bigger than the number of compact $G$-orbits.
\end{yuppi}
Observe that $\mu_{\mathfrak p} (\OO)$ is a $K$-orbit but it is not true in general that $\mu_{\mathfrak p}$ defines a diffeomporhism between $\OO$ and $\mup(\OO)$, without the assumption that $G$ is a complex reductive group. Therefore, Theorem $1.2$ in \cite[pag. $582$]{LA} does not apply in our context.

Let $\xi \in \mathfrak u$. The standard notation for parabolic subgroups of complex reductive groups, see for instance \cite{Kirwan},
is given by
\[
U^{\C}(\xi)=\{g\in U^\C:\, \lim_{t\to -\infty} \exp(it\xi)g\exp(-it\xi)\, \mathrm{exists}\, \}.
\]
It is well-known that $U^{\C}(\xi)$ is connected and it contains a Borel subgroup, that is, a maximal solvable subgroup of $U^\C$ \cite{akhiezer}.
Hence, if $\beta \in i \mathfrak u$, then $U^{\C}(-i\beta)$ corresponds to $(U^{\C})^{\beta+}$ in our notation. If $\tilde \OO$ is a compact orbit of $U^\C$ then it is a complex $U$ orbit  \cite{heinzner-schwarz-stoetzel} and so a flag manifold \cite{Guillemin2}. Since
\[
\mathrm{max}_{\tilde \OO}(\beta)=\left\{p\in \tilde \OO:\, \mathrm{max}_{z\in \tilde \OO} \mu^{-i\beta} =\mu^{-i\beta} (p) \right\},
\]
it follows that $\mathrm{max}_{\tilde \OO}(\beta)$ is connected \cite{Atiyah,Guillemin}. Hence the following result, see also \cite{biliotti-ghigi-heinzner-1-preprint}, holds.
\begin{corollary}
The number of compact $(U^\C)^{\beta +}$-orbits is equal to the number of compact $U^\C$-orbits. Moreover, any closed $(U^\C)^{\beta +}$-orbit arises as $\mathrm{max}_{\tilde \OO}(\beta)$, where $\tilde \OO$ is a compact $U^\C$-orbit.
\end{corollary}
Assume that $G$ is a real form of $U^\C$. Assume there exists $p\in Z$ such that $X=U^\C\cdot p$ is compact. If $Z$ is compact, then the $U^\C$ orbit throughout the maximum of the norm square function  $\parallel \mu \parallel^2$ is a compact orbit and so it is a flag manifold \cite{heinzner-schwarz-stoetzel}. It is well-known that $G$ has a unique closed orbit $\OO$ in $X$. This is an old result of Wolf \cite{Wolf}, see also \cite{heinzner-schwarz-stoetzel}. In this setting, we prove the following result.
\begin{yuppi}
The set $\mathrm{max}_{\OO}(\beta)$ is the unique closed orbit of $G^{\beta+}$ acting on $X$. This orbit is connected and it is a $(K^{\beta})^o$ orbit.
\end{yuppi}
As a consequence of the proof we obtain the following result.
\begin{proposition}
Let $\mathfrak a \subset \liep$ be an Abelian subalgebra. Let $\mu_{\mathfrak a} :X \longrightarrow \mathfrak a$ the corresponding $A=\exp(\mathfrak a)$-gradient map. Then $\mu_{\mathfrak a}(X)=\mu_{\mathfrak a}(\OO)$.
\end{proposition}
It is well-known that both  $\mu_{\mathfrak a}(X)$ and $\mu_{\mathfrak a}(\OO)$ are polytope \cite{heinzner-schwarz-stoetzel}. The above result tells us that $\OO$ captures much of the informations of $\mu_{\mathfrak a}$. Note that if $\lia \subset \liep$ is a maximal Abelian subalgebra then by a beautiful Theorem of Kostant \cite{kostant-convexity}, keeping in mind that $\mu_{\mathfrak a}=\pi_\lia \circ \mu_{\mathfrak p}$ and $\mup(\OO)$ is a $K$ orbit in $\liep$, the set $\mu_{\mathfrak a}(X)$ is the convex hull of an orbit of the Weyl group $W(\lieg,\lia)=\{\mathrm{Ad}(k):\, k\in K,\, \mathrm{Ad}(k)(\lia)=\lia\}$ (see \cite{knapp-beyond} for more details on Weyl group).
\section{Preliminaries}
\subsection{Convex geometry}
\label{pre-convex}
In this section, some definitions and results in convex geometry are recalled. The reader can see e.g. \cite{schneider-convex-bodies} and \cite{biliotti-ghigi-heinzner-1-preprint} for further details on the topic. Let $V$ be a real vector space with a scalar product $\scalo$ and let $E\subset V$ be a \changed{compact} convex subset. The \emph{relative interior} of $E$, denoted $\relint E$, is the interior of $E$ in its affine hull. For $a,b\in E,$ denote the close segment joining $a$ and $b$ by $[a,b]$. Then, a face of $E$ is a convex subset $F$ of $E$ such that if $a,b\in E$ and $\relint[a,b]\cap F\neq \vacuo$, then
 $[a,b]\subset F$.  The \emph{extreme points} of $E$ denoted by $\ext E$ are the points
 $a\in E$ such that $\{a\}$ is a face. Since  $E$ is compact the faces
 are closed \cite[p. 62]{schneider-convex-bodies}. The empty set and $E$ are faces of $E:$ the other faces are called \enf{proper}.
 \begin{definition}
  The support function of $E$ is defined by the function $ h_E : V \ra \R$, $
 h_E(u) = \max \{\sx x, u \xs : x\in E\}$.  If $ u\neq 0$, the
 hyperplane $H(E, u) : = \{ x\in E : \sx x, u \xs = h_E(u)\}$ is
 called the supporting hyperplane of $E$ for $u$.
The set
   \begin{gather}
     \label{def-exposed}
     F_u (E) : = E \cap H(E,u)
   \end{gather}
   is a face and it is called the \enf{exposed face} of $E$ defined by
   $u$.
 \end{definition}
In general not all faces of a convex subset are exposed. For instance, consider the convex hull of a closed disc and a point outside the disc: the resulting convex set is the union of the
 disc and a triangle. The two vertices of the triangle that lie on the
 boundary of the disc are non-exposed 0-faces.

The following result is probably well-known. A proof is given in \cite{LBS}. For sake of completeness we give a proof.
\begin{proposition}\label{convex-criterium}
Let $C_1 \subseteq C_2$ be two compact convex set of $V$. Assume that for any $\beta \in V$ we have
\[
\mathrm{max}_{y\in C_1} \langle y , \beta \rangle=\mathrm{max}_{y\in C_2} \langle y , \beta \rangle.
\]
Then $C_1=C_2$.
\end{proposition}
\begin{proof}
We may assume without loss of generality  that the affine hull of $C_2$ is $V$.
Assume by contradiction that $C_1 \subsetneq C_2$. Since $C_1$ and $C_2$ are both compact, it follows that there exists $p\in \partial C_1$ such that $p\in \stackrel{o}{C_2}$. Since every face of a convex compact set is contained in an exposed face \cite{schneider-convex-bodies}, there exists $\beta \in V$ such that
\[
\mathrm{max}_{y\in C_1} \langle y , \beta \rangle=\langle p, \beta \rangle.
\]
This means the linear function $x\mapsto \langle x, \beta \rangle$ restricted on $C_2$ achieves its maximum at an interior point which is a contradiction.
\end{proof}
\subsection{Gradient map}
\label{subsection-gradient-moment}
Let $(Z, \om)$ be a \Keler manifold. Let $U^\C$ acts
holomorphically on $Z$ i.e., the action $U^\C \times Z \to Z$ is holomorphic. Assume that $U$ preserves $\om$ and that there is a $U$-equivariant
momentum map $\mu: Z \ra \liu$.  If $\xi \in \liu$, we denote by $\xi_Z$
the induced vector field on $Z$ and we let $\mu^\xi \in \cinf(Z)$ be
the function $\mu^\xi(z) := \sx \mu(z),\xi\xs$. By definition, we have
$$
d\mu^\xi =
i_{\xi_Z} \om.
$$

Let $G \subset U^\C$ be a compatible subgroup of $U^\C$.
For $x\in Z$, let $\mu_\mathfrak{p}^\beta(x)$ denote $-i$ times the component of $\mu (x)$ along $\beta$ in the direction of $i\mathfrak{p},$ i.e.,
\begin{equation}\label{mu3}
\mu_\mathfrak{p}^\beta(x) := \langle \mu_\mathfrak{p}(x), \beta \rangle = -\langle \mu(x), i\beta \rangle = \mu^{-i\beta}
\end{equation} for any $\beta \in \mathfrak{p}.$ Here $\langle\cdot , \cdot\rangle$ denotes $K$-invariant inner product on $\mathfrak{p} \subset i\mathfrak{u}.$ Then, the map defined by $\mu_\liep : Z \ra \liep$ is called the \emph{gradient map}.
Let $\mup^\beta \in \cinf(Z)$ be the function $ \mup^\beta(z) = \sx
\mup(z) , \beta\xs = \mu^{-i\beta}(z)$.  Let $\metrica$ be the \Keler
metric associated to $\om$, i.e. $(v, w) = \om (v, Jw),$ for all $z\in Z$ and $v,w\in T_zZ$ where $J$ denotes the complex structure on $TZ$. Then
$\beta_Z$ is the gradient of $\mup^\beta$.

For the rest of this paper, we assume that $G$ is connected and we fix a $G$-invariant locally closed submanifold $X$ of $Z.$ From now on, we denote the restriction of $\mu_\mathfrak{p}$ to $X$ by $\mu_\mathfrak{p}.$ Then
$$
\text{grad}\mu_\mathfrak{p}^\beta = \beta_X,
$$ where grad is computed with respect to the induced Riemannian metric on $X.$

Let $\beta \in \mathfrak{p}$. It is well-known that $G^\beta$ is compatible and
\[
G^\beta=K^\beta \text{exp}  (\liep^\beta),
\]
where $K^\beta=K\cap G^\beta=\{g\in K:\, \mathrm{Ad}(g)(\beta)=\beta\}$ and
$\liep^\beta=\{v\in \liep:\, [v,\beta]=0\}$ (see \cite{knapp-beyond}).
\begin{corollary} \label{slice-cor} If $x \in X$ and $\mup(x) = \beta$,
  there are a $G^\beta$-invariant decomposition $T_xX = \lieg^\beta
  \cd x \, \oplus W$, open $G^\beta$-invariant subsets $S \subset W$,
  $\Omega \subset X$ and a $G^\beta$-equivariant diffeomorphism $\Psi
  : G^\beta \times^{G_x}S \ra \Omega$, such that $0\in S, x\in \Omega$
  and $\Psi ([e, 0]) =x$.
\end{corollary}
\begin{proof}
See \cite[p. 169]{heinzner-schwarz-stoetzel}.
\end{proof}
$\beta_X$ is a vector field on
  $X$, i.e. a section of $TX$. For $x\in X$, the differential is a map
  $T_x X \ra T_{\beta_X(x)}(TX)$. If $\beta_X(x) =0$, there is a
  canonical splitting $T_{\beta_X(x)}(TX) = T_x X \oplus
  T_x X$. Accordingly $d\beta_X(x)$ splits into a horizontal and a
vertical part. The horizontal part is the identity map. We denote the
vertical part by $d\beta_X(x)$.  It belongs to $\End(T_xX)$.  Let
$\{\phi_t=\text{exp} (t\beta)\} $ be the flow of $\beta_X$.  There
  is a corresponding flow on $TX$. Since $\phi_t(x)=x$, the flow on
  $TX$ preserves $T_xX$ and there it is given by $d\phi_t(x) \in
  \Gl(T_xX)$.  Thus we get a linear $\R$-action on $T_x X$ with
  infinitesimal generator $d\beta_X(x) $.
\begin{corollary}
  \label{slice-cor-2}
  If $\beta \in \liep $ and $x \in X$ is a critical point of $\mupb$,
  then there are open invariant neighbourhoods $S \subset T_x X$ and
  $\Omega \subset X$ and an $\R$-equivariant diffeomorphism $\Psi : S
  \ra \Omega$, such that $0\in S, x\in \Omega$, $\Psi ( 0) =x$. (Here
  $t\in \R$ acts as $d\phi_t(x)$ on $S$ and as $\phi_t$ on $\Omega$.)
\end{corollary}
\begin{proof}
See \cite{heinzner-schwarz-stoetzel}.
\end{proof}
Let $x \in
\Crit(\mu^\beta_\liep)$. Let $D^2\mup^\beta(x) $ denote the Hessian,
which is a symmetric operator on $T_xX$ such that
\begin{gather*}
  ( D^2 \mup^\beta(x) v, v) = \frac{\mathrm d^2}{\mathrm
    dt^2} 
  (\mup^\beta\circ \ga)(0)
\end{gather*}
where $\ga$ is a smooth curve, $\ga(0) = x$ and $ \dot{\ga}(0)=v$.
Denote by $V_-$ (respectively $V_+$) the sum of the eigenspaces of the
Hessian of $\mupb$ corresponding to negative (resp. positive)
eigenvalues. Denote by $V_0$ the kernel.  Since the Hessian is
symmetric we get an orthogonal decomposition
\begin{gather}
  \label{Dec-tangente}
  T_xX = V_- \oplus V_0 \oplus V_+.
\end{gather}
Let $\alfa : G \ra X$ be the orbit map: $\alfa(g) :=gx$.  The
differential $d\alfa_e$ is the map $\xi \mapsto \xi_X(x)$.
\begin{proposition}
    \label{tangent}
  If $\beta \in \liep$ and $x \in \Crit(\mu^\beta_\liep)$ then
  \begin{gather*}
    D^2\mup^\beta(x) = d \beta_X(x).
  \end{gather*}
  Moreover $d\alfa_e (\lier^{\beta\pm} ) \subset V_\pm$ and $d\alfa_e(
  \lieg^\beta) \subset V_0$.  If $X$ is $G$-homogeneous these are
  equalities.
\end{proposition}
\begin{proof}
See \cite[Prop. 2.5]{heinzner-schwarz-stoetzel} and \cite{LA}.
\end{proof}
\begin{corollary}
  \label{MorseBott}
  For every $\beta \in \liep$, $\mupb$ is a Morse-Bott function.
\end{corollary}
Using an $\mathrm{Ad}(K)$-invariant  inner product of  $\liep$, we define $\nu_\liep (z):=\frac{1}{2}\parallel \mup(z)\parallel^2$. The function $\nu_\liep$ is $K$-invariant and it is called \emph{the norm square function}.
The following result is proved in \cite{heinzner-schwarz-stoetzel} (see Corollary 6.11 and Corollary 6.12 p. $21$).
\begin{proposition}\label{heinzner-maximun}
Let $x\in M$. Then:
\begin{itemize}
\item if $\nu_\liep$ restricted to $G\cdot x$ has a local maximum at $x$, then $G\cdot x=K\cdot x$
\item if $G\cdot x$ is compact, then $G\cdot x =K\cdot x$
\end{itemize}
\end{proposition}
A strategy to analyzing the $G$ action on $M$ is to view $\nu_\liep$ as generalized Morse function. In \cite{heinzner-schwarz-stoetzel} the authors proved the existence of a smooth $G$-invariant stratification of $M$ and they studied its properties.
\section{Closed Orbit of Parabolic Subgroups}
Let $(Z,\omega)$ be a \Keler manifold and $U^\C$ acts holomorphically on $Z$ with a momentum map $\mu : Z \to \mathfrak{u}.$ Let $G\subset U^\C$ be a closed compatible subgroup. $G = K\text{exp}(\mathfrak{p}),$ where $K := G\cap U$ is a maximal compact subgroup of $G$ and $\mathfrak{p} := \mathfrak{g}\cap \text{i}\mathfrak{u};$ $\mathfrak{g}$ is the Lie algebra of $G.$

\textbf{Remarks:} Suppose $X\subset Z$ is $G$-stable locally closed real submanifold of $Z$ with the gradient map $\mu_\mathfrak{p} : X\to \mathfrak{p}.$ Let $Q \subset G$ be a parabolic subgroup. The following facts are easy to check:

\begin{enumerate}
    \item If $Q\cdot p$ is compact, then $G\cdot p$ is closed since $G = KQ;$

    \item let $\OO$ be a compact $G$-orbit. By Proposition \ref{heinzner-maximun}, it follows that $\OO = G\cdot p = K\cdot p$. Since $\mup$ is $K$-equivariant, the restricted gradient map  $\mu_\mathfrak{p} : K\cdot p \to K\cdot \mu_\mathfrak{p}(x)$ is a smooth $K$-equivariant submersion.
\end{enumerate}
Let $\beta \in \mathfrak{p}$ and let $$Y = \{z\in X : \mathrm{max}_{x\in X}\mu_\mathfrak{p}^\beta(x) = \mu_\mathfrak{p}^\beta(z)\}.$$ Assume that $Y$ is not empty. By Corollary \ref{slice-cor-2}, $Y$ is a smooth, possibly disconnected, submanifold of $X$.
\begin{lemma}\label{lemmm}
$Y$ is $G^{\beta +}$ invariant.
\end{lemma}
\begin{proof}

Let $g\in G$ and let $\xi \in \liep$. It is easy to check that
\[
(\mathrm{d} g)_p (\xi_X)=(\mathrm{Ad}(g)(\xi) )_X (gp ),
\]
and so $G^\beta$ preserves $X^\beta$. We claim that $Y$ is $G^\beta$-stable. In fact, $G^\beta = K^\beta\text{exp}(\mathfrak{p}^\beta)$ and $Y$ is $K^\beta$ invariant by $K$-invariant property of the gradient map. For each $y\in Y,$ let $\xi\in \liep^\beta$ and let $\gamma(t)= \text{exp}(t\xi)\cdot y.$ Since $\beta_X (\gamma(t))=0$ it follows that $\mup^\beta (\gamma(t) )$ is constant and so $\text{exp} (t\xi)\cdot y\in Y.$ Now, $G^{\beta+} = G^\beta R^{\beta+}$ where $R^{\beta+}$ is connected and then unipotent radical of $G^{\beta+}$. By Proposition \ref{tangent}, $\lier^{\beta+} \subset V_+$ and so $\lier^{\beta+}\cdot z \subset G_z$ for all $z\in Y.$ Since $R^{\beta+}$ is connected, this implies $R^{\beta+}$ does not act on $Y$ and the result follows.
\end{proof}
\begin{lemma}\label{gm}
Let $\OO$ be a compact $G$-orbit. Let $\beta \in \mathfrak{p}.$ If $x\in \OO$ is a local maximum of $\mu_\mathfrak{p}^\beta : \OO \to \mathbb{R},$ then $x$ is a global maximum of $\mu_\mathfrak{p}^\beta : \OO \to \mathbb{R}.$
\end{lemma}
\begin{proof}
If $x\in \OO$ is a local maximum of $\mu_\mathfrak{p}^\beta,$ then $\mu_\mathfrak{p}(x)$ is a local maximum of the height function
$$
K \cdot \mu_\mathfrak{p}(x) \to \mathbb{R}, \quad z \mapsto \langle z, \beta \rangle.
$$ But it was noted in the proof of Proposition 3.9 in \cite{LA} that $\langle \cdot, \beta\rangle$ has only global maximum when restricted to $K \cdot \mu_\mathfrak{p}(x)$. Then a local maximum is a global maximum and this implies that $\mu_\mathfrak{p}(x)$ is a global maximum of the height function $\langle \cdot, \beta\rangle.$
Since
$$
\mathrm{max}_{\OO} (\beta)=  \mathrm{max} \{\langle z, \beta \rangle,\, z\in K\cdot \mup(x)\},
$$ $x$ is a global maximum of $\mu_\mathfrak{p}^\beta.$
\end{proof}

\begin{proposition}
\label{Maxx}
Let $p\in \OO$ be such that $G^{\beta+}\cdot p$ is closed. Then $G^{\beta+}\cdot p$ is a finite union of connected components of
$
\mathrm{max}_{\OO}(\beta).
$
\end{proposition}

\begin{proof}
Since $G^{\beta+} \cdot p$ is compact, $\mu_\mathfrak{p}^\beta|_{G^{\beta+}\cdot p}$ has a maximum. Let $q\in G^{\beta+}\cdot p$ denote a maximum of $\mu_\mathfrak{p}^\beta|_{G^{\beta+}\cdot p}$. By Proposition \ref{tangent}, $q$ is a $R^{\beta+}$ fixed point. Applying again, Proposition \ref{tangent}, keeping in mind that $\OO$ is $G$ homogeneous, $q$ is a local maximum of $\mu_\mathfrak{p}^\beta : \OO \to \mathbb{R.}$ By Lemma \ref{gm}, $q$ is a global maximum of $\mu_\mathfrak{p}^\beta.$ By Lemma \ref{lemmm}, the unipotent group $R^{\beta+}$ acts trivially on $\mathrm{max}_{\OO}(\beta)$ and $G^{\beta+}\cdot p \subset \mathrm{max}_{\OO}(\beta).$

Let $x \in \mathrm{max}_{\OO}(\beta),$ By Proposition \ref{tangent} and Corollary \ref{MorseBott}, keeping in mind that $\OO$ is $G$ homogeneous and $R^{\beta+}$ acts trivially on $\mathrm{max}_{\OO}(\beta)$, it follows that
$T_x\mathrm{max}_{\OO}(\beta) = T_x G^\beta \cdot x.$ By Lemma \ref{lemmm} $(G^\beta)^o$ preserves any connected component of $\mathrm{max}_{\OO}(\beta).$ Moreover, the restriction of $\mu_\mathfrak{p}$ to any connected component of $\mathrm{max}_{\OO}(\beta)$ defines the gradient map of $(G^\beta)^o$, see \cite{heinzner-schwarz-stoetzel}. By Proposition \ref{heinzner-maximun} it follows that $(G^\beta)^o$ has a closed orbit on any connected component of $\mathrm{max}_{\OO}(\beta).$ Since any $(G^{\beta})^o$ orbit is open in  $\mathrm{max}_{\OO}(\beta),$ it follows that the connected component of $\mathrm{max}_{\OO}(\beta)$ containing $x$ is $(G^\beta)^o$ homogeneous. The connected components of $G^\beta$ are finite and intersect the connected components of $K^\beta$. Therefore, keeping in mind that $R^{\beta+}$ acts trivially on $\mathrm{max}_{\OO}(\beta)$, $G^{\beta+}\cdot x$ is a finite union of connected components of $\mathrm{max}_{\OO}(\beta).$ The same result holds for $G^{\beta+}\cdot p$, concluding the proof.
\end{proof}
\begin{corollary}
Let $x\in \mathrm{max}_{\OO}(\beta).$ Then $G^{\beta+} \cdot x$ is closed and a finite union of connected components of $\mathrm{max}_{\OO}(\beta).$
\end{corollary}
Summing up, we have proved our first main result.
\begin{theorem}\label{summing}
Let $\beta \in \mathfrak{p}.$ Then:
\begin{itemize}
    \item if $G^{\beta+}\cdot x$ is compact, then $\OO = G\cdot x$ is compact and $G^{\beta+}\cdot x$ is a finite union of connected components of $\mathrm{max}_{\OO}(\beta):$

    \item if $\OO$ is a compact $G$-orbit, then $\mathrm{max}_{\OO}(\beta)$ is a finite union of compact $G^{\beta+}$-orbits.
\end{itemize}
In particular, the number of compact $G^{\beta +}$-orbits is equal or bigger than the number of compact $G$-orbits.
\end{theorem}
Let $Q\subset U^\C$ be a parabolic subgroup. There exists $\beta \in i \mathfrak{u}$ such that
$
Q = (U^\C)^{\beta+}.
$
If $\tilde \OO$ is a compact $U^\C$-orbit, then it is a complex $U$-orbit and so a flag manifold \cite{Guillemin2}. By definition of the gradient map,
\[
\mathrm{max}_{\tilde \OO}(\beta)=\mathrm{max}\{p\in \tilde \OO:\, \langle \mu(p), -i\beta\rangle=\mathrm{max}_{p\in \tilde \OO} \mu^{-i\beta} \}
 \]
and so it is connected \cite{Atiyah,Guillemin}. This means that $(U^\C)^{\beta +}$ has a unique closed orbit in $\tilde \OO.$
\begin{corollary}\label{complex-parabolic}
The number of compact $(U^\C)^{-i\beta +}$-orbits is equal to the number of compact $U^\C$-orbits. Any closed $(U^\C)^{-i\beta +}$-orbit arises as $\mathrm{max}_{\OO}(\beta)$, where $\tilde \OO$ is a compact $U^\C$-orbit.
\end{corollary}
Assume that $G$ is a real form of $U^\C.$ Then $\mathfrak{g} = \mathfrak{k}\oplus \mathfrak{p}$, $\mathfrak{u} = \mathfrak{k}\oplus i\mathfrak{p}$ and so
 $\mathfrak{g}^\C = \mathfrak{u}^\C$ (see \cite{heinzner-stoetzel,knapp-beyond}).

Assume that there exists $x\in Z$ such that $U^\C \cdot x$ is compact. Then $U^\C \cdot x =U\cdot x$ and so a flag manifold. The following result is essentially an old Theorem of Wolf \cite{Wolf}, see also \cite{heinzner-schwarz-stoetzel}.
\begin{theorem}[Wolf]\label{Unique closed}
There exists a unique closed $G$-orbit in $U^\C \cdot x.$
\end{theorem}
\begin{proof}
Let $G_{ss}$ denote the connected subgroup of $G$  with Lie algebra $[\mathfrak g,\mathfrak g]$. Then $G_{ss}$ is closed, compatible and $G = Z(G)^o \cdot G_{ss},$ where $Z(G)^o$ is the connected component of the center \cite[p.442]{knapp-beyond}. By Proposition \ref{heinzner-maximun}, $G$ has a closed orbit in $U^\C\cdot x.$ The center of $U$ does not act on $U\cdot x$ and $G_{ss}$ is a real from of $(U^\C)_{ss} = (U_{ss})^\C$. By a Theorem of Wolf \cite{Wolf}, $G_{ss}$ has a unique closed orbit in $U^\C\cdot x.$ On the other hand, by Proposition \ref{heinzner-maximun} it follows that $G_{ss}$ has  a closed orbit on any closed orbit of $G$.   Therefore, $G$ has a unique closed orbit in $U^\C\cdot x.$
\end{proof}
Let $\OO$ denote the unique compact $G$-orbit in $U^\C \cdot x.$ Let $\beta \in \liep$. We denote by
\[
\mathrm{max}_{U^\C \cdot x} (\beta)=\left\{p\in U^\C \cdot x:\, \mu_{\mathfrak p}^{\beta}(p)=\mathrm{max}_{p\in U^\C \cdot x} \mu_{\mathfrak p}^\beta  \right\}
\]
\begin{lemma}
For any $\beta \in \mathfrak{p},$  $\mathrm{max}_{U^\C \cdot x} (\beta) \cap \OO \neq \emptyset$. Hence
$$
\mathrm{max}_{U^\C\cdot x } (\beta) \cap \OO =\mathrm{max}_{\OO} (\beta).
$$
\end{lemma}

\begin{proof}
Since
$$
\mathrm{max}_{U^\C\cdot x} (\beta ) = \mathrm{max}\left\{p\in U^\C \cdot x:\, \langle \mu(p), -i\beta\rangle=\mathrm{max}_{p\in U^\C \cdot x} \mu^{-i\beta}\right\},
$$
by Corollary \ref{complex-parabolic} it follows that $\mathrm{max}_{U^\C \cdot x} (\beta)$ is the unique closed orbit of $(U^\C)^{\beta+}.$ Since
$$
G^{\beta+} = G\cap(U^\C)^{\beta+}
$$
we get $R^{\beta+}\subset R((U^\C)^{-i\beta+})$ and so $R^{\beta+}$ acts trivially on $\mathrm{max}_{U^\C\cdot x}\mu_\mathfrak{p}^\beta$. Applying the same arguments of Proposition \ref{Maxx}, it follows that $G^{\beta+}$ has a closed orbit in $\mathrm{max}_{U^\C\cdot x}\mu_\mathfrak{p}^\beta$. By Theorem \ref{Unique closed} we get $\mathrm{max}_{U^\C\cdot x} (\beta) \cap \OO \neq \emptyset$ and the result follows.
\end{proof}
As a consequence, we obtain the following result.
\begin{proposition}
Let $\mathfrak{a}\subset \mathfrak{p}$ be an Abelian subalgebra. Then
$$
\mu_\mathfrak{a}(U^\C\cdot x) = \mu_\mathfrak{a}(\OO).
$$
\end{proposition}
\begin{proof}
It is well-known that both $\mu_\mathfrak{a}(U^\C\cdot x)$ and $\mu_\mathfrak{a}(\OO)$ are polytope \cite{kostant-convexity,heinzner-schwarz-stoetzel}. Applying the above Lemma and Proposition \ref{convex-criterium}, we get $\mu_\mathfrak{a}(U^\C\cdot x) = \mu_\mathfrak{a}(\OO).$
\end{proof}
Now we are ready to prove our second main result.
\begin{theorem}
The set $\mathrm{max}_{\OO} (\beta)$ is the unique closed orbit of $G^{\beta +}$ in $\OO$. Moreover, it is connected and a $(K^\beta)^o$-orbit.
\end{theorem}
\begin{proof}
Let $(G^\beta)^o_{ss}$ denote the connected subgroup whose Lie algebra is $[\mathfrak g^{\beta},\mathfrak g^{\beta}]$. It is closed, semisimple and compatible \cite{knapp-beyond}. By Lemma \ref{lemmm} it preserves any connected components of $\mathrm{max}_{\OO}(\beta)$. By Proposition \ref{heinzner-maximun},  $(G^\beta)^o_{ss}$ has  a closed orbit on any connected component $\mathrm{max}_{\OO} (\beta)$. On the other hand, $\mathrm{max}_{U^\C \cdot x} (\beta)$ is connected and, by Proposition \ref{tangent}, is a closed orbit of $(U^\C)^{\beta}$. Note that $(G^\beta)^o_{ss}$ is a real form of $(U^{\beta})_{ss}$ and $\mathrm{max}_{U^\C \cdot x} (\beta)$, keeping in mind that it is a flag manifold and the center of $(U^\C )^\beta$ does not act on it, is a compact $(U^{\beta})_{ss}$ orbit. Applying a Theorem of Wolf \cite{Wolf} it follows that $(G^\beta)^o_{ss}$ has a unique closed orbit in $\mathrm{max}_{U^\C \cdot x} (\beta)$. Since both $R^{\beta+}$ and the center of $(G^{\beta})^o$ act trivially on  $\mathrm{max}_{U^\C \cdot x} (\beta)$, the  unique closed orbit of $(G^\beta)^o_{ss}$ is contained in a closed orbit of $G^{\beta+}$ and so it is contained in $\OO$. By Theorem \ref{summing}, this orbit is contained in $\mathrm{max}_{\OO}(\beta)$. Since $(G^\beta)^o_{ss}$ preserves $\mathrm{max}_{\OO} (\beta)$ and it has a closed orbit on any connected component of $\mathrm{max}_{\OO} (\beta)$ it follows that $\mathrm{max}_{\OO} (\beta)$ is connected. This means $\mathrm{max}_{\OO} (\beta)$ is the unique closed orbit $G^{\beta +}$. In particular, keeping in mind $\mathrm{max}_{\OO} (\beta)$ is $(G^\beta)^o$ homogeneous, applying Proposition \ref{heinzner-maximun} we get $\mathrm{max}_{\OO} (\beta)$ is a $(K^\beta)^o$ orbit concluding the proof.
\end{proof}

\end{document}